\documentclass{amsart}
\usepackage{amssymb}
\title{The Combinatorial Nullstellens\"atze Revisited}
\author{Pete L. Clark}
\begin{document}
\newtheorem{lemma}{Lemma}
\newtheorem{prop}[lemma]{Proposition}
\newtheorem{cor}[lemma]{Corollary}
\newtheorem{thm}[lemma]{Theorem}
\newtheorem{ques}{Question}
\newtheorem{quest}[lemma]{Question}
\newtheorem{conj}[lemma]{Conjecture}
\newtheorem{fact}[lemma]{Fact}
\newtheorem*{mainthm}{Main Theorem}
\newtheorem*{unthm}{Theorem}
\newtheorem{obs}[lemma]{Observation}
\newtheorem{hint}{Hint}
\newtheorem{remark}[lemma]{Remark}
\newtheorem{example}[lemma]{Example}

\begin{abstract}
We revisit and further explore the celebrated Combinatorial Nullstellens\"atze of N. Alon in several different directions. 
\end{abstract}

\maketitle
\newcommand{\pp}{\mathfrak{p}}
\newcommand{\DD}{\mathcal{D}}
\newcommand{\F}{\ensuremath{\mathbb F}}
\newcommand{\Fp}{\ensuremath{\F_p}}
\newcommand{\Fl}{\ensuremath{\F_l}}
\newcommand{\Fpbar}{\overline{\Fp}}
\newcommand{\Fq}{\ensuremath{\F_q}}
\newcommand{\PP}{\mathbb{P}}
\newcommand{\PPone}{\mathfrak{p}_1}
\newcommand{\PPtwo}{\mathfrak{p}_2}
\newcommand{\PPonebar}{\overline{\PPone}}
\newcommand{\N}{\ensuremath{\mathbb N}}
\newcommand{\Q}{\ensuremath{\mathbb Q}}
\newcommand{\Qbar}{\overline{\Q}}
\newcommand{\R}{\ensuremath{\mathbb R}}
\newcommand{\Z}{\ensuremath{\mathbb Z}}
\newcommand{\SSS}{\ensuremath{\mathcal{S}}}
\newcommand{\Rn}{\ensuremath{\mathbb R^n}}
\newcommand{\Ri}{\ensuremath{\R^\infty}}
\newcommand{\C}{\ensuremath{\mathbb C}}
\newcommand{\Cn}{\ensuremath{\mathbb C^n}}
\newcommand{\Ci}{\ensuremath{\C^\infty}}\newcommand{\U}{\ensuremath{\mathcal U}}
\newcommand{\gn}{\ensuremath{\gamma^n}}
\newcommand{\ra}{\ensuremath{\rightarrow}}
\newcommand{\fhat}{\ensuremath{\hat{f}}}
\newcommand{\ghat}{\ensuremath{\hat{g}}}
\newcommand{\hhat}{\ensuremath{\hat{h}}}
\newcommand{\covui}{\ensuremath{\{U_i\}}}
\newcommand{\covvi}{\ensuremath{\{V_i\}}}
\newcommand{\covwi}{\ensuremath{\{W_i\}}}
\newcommand{\Gt}{\ensuremath{\tilde{G}}}
\newcommand{\gt}{\ensuremath{\tilde{\gamma}}}
\newcommand{\Gtn}{\ensuremath{\tilde{G_n}}}
\newcommand{\gtn}{\ensuremath{\tilde{\gamma_n}}}
\newcommand{\gnt}{\ensuremath{\gtn}}
\newcommand{\Gnt}{\ensuremath{\Gtn}}
\newcommand{\Cpi}{\ensuremath{\C P^\infty}}
\newcommand{\Cpn}{\ensuremath{\C P^n}}
\newcommand{\lla}{\ensuremath{\longleftarrow}}
\newcommand{\lra}{\ensuremath{\longrightarrow}}
\newcommand{\Rno}{\ensuremath{\Rn_0}}
\newcommand{\dlra}{\ensuremath{\stackrel{\delta}{\lra}}}
\newcommand{\pii}{\ensuremath{\pi^{-1}}}
\newcommand{\la}{\ensuremath{\leftarrow}}
\newcommand{\gonem}{\ensuremath{\gamma_1^m}}
\newcommand{\gtwon}{\ensuremath{\gamma_2^n}}
\newcommand{\omegabar}{\ensuremath{\overline{\omega}}}
\newcommand{\dlim}{\underset{\lra}{\lim}}
\newcommand{\ilim}{\operatorname{\underset{\lla}{\lim}}}
\newcommand{\Hom}{\operatorname{Hom}}
\newcommand{\Ext}{\operatorname{Ext}}
\newcommand{\Part}{\operatorname{Part}}
\newcommand{\Ker}{\operatorname{Ker}}
\newcommand{\im}{\operatorname{im}}
\newcommand{\ord}{\operatorname{ord}}
\newcommand{\unr}{\operatorname{unr}}
\newcommand{\B}{\ensuremath{\mathcal B}}
\newcommand{\Ocr}{\ensuremath{\Omega_*}}
\newcommand{\Rcr}{\ensuremath{\Ocr \otimes \Q}}
\newcommand{\Cptwok}{\ensuremath{\C P^{2k}}}
\newcommand{\CC}{\ensuremath{\mathcal C}}
\newcommand{\gtkp}{\ensuremath{\tilde{\gamma^k_p}}}
\newcommand{\gtkn}{\ensuremath{\tilde{\gamma^k_m}}}
\newcommand{\QQ}{\ensuremath{\mathcal Q}}
\newcommand{\I}{\ensuremath{\mathcal I}}
\newcommand{\sbar}{\ensuremath{\overline{s}}}
\newcommand{\Kn}{\ensuremath{\overline{K_n}^\times}}
\newcommand{\tame}{\operatorname{tame}}
\newcommand{\Qpt}{\ensuremath{\Q_p^{\tame}}}
\newcommand{\Qpu}{\ensuremath{\Q_p^{\unr}}}
\newcommand{\scrT}{\ensuremath{\mathfrak{T}}}
\newcommand{\That}{\ensuremath{\hat{\mathfrak{T}}}}
\newcommand{\Gal}{\operatorname{Gal}}
\newcommand{\Aut}{\operatorname{Aut}}
\newcommand{\tors}{\operatorname{tors}}
\newcommand{\Zhat}{\hat{\Z}}
\newcommand{\linf}{\ensuremath{l_\infty}}
\newcommand{\Lie}{\operatorname{Lie}}
\newcommand{\GL}{\operatorname{GL}}
\newcommand{\End}{\operatorname{End}}
\newcommand{\aone}{\ensuremath{(a_1,\ldots,a_k)}}
\newcommand{\raone}{\ensuremath{r(a_1,\ldots,a_k,N)}}
\newcommand{\rtwoplus}{\ensuremath{\R^{2  +}}}
\newcommand{\rkplus}{\ensuremath{\R^{k +}}}
\newcommand{\length}{\operatorname{length}}
\newcommand{\Vol}{\operatorname{Vol}}
\newcommand{\cross}{\operatorname{cross}}
\newcommand{\GoN}{\Gamma_0(N)}
\newcommand{\GeN}{\Gamma_1(N)}
\newcommand{\GAG}{\Gamma \alpha \Gamma}
\newcommand{\GBG}{\Gamma \beta \Gamma}
\newcommand{\HGD}{H(\Gamma,\Delta)}
\newcommand{\Ga}{\mathbb{G}_a}
\newcommand{\Div}{\operatorname{Div}}
\newcommand{\Divo}{\Div_0}
\newcommand{\Hstar}{\cal{H}^*}
\newcommand{\txon}{\tilde{X}_0(N)}
\newcommand{\sep}{\operatorname{sep}}
\newcommand{\notp}{\not{p}}
\newcommand{\Aonek}{\mathbb{A}^1/k}
\newcommand{\Wa}{W_a/\mathbb{F}_p}
\newcommand{\Spec}{\operatorname{Spec}}

\newcommand{\abcd}{\left[ \begin{array}{cc}
a & b \\
c & d
\end{array} \right]}

\newcommand{\abod}{\left[ \begin{array}{cc}
a & b \\
0 & d
\end{array} \right]}

\newcommand{\unipmatrix}{\left[ \begin{array}{cc}
1 & b \\
0 & 1
\end{array} \right]}

\newcommand{\matrixeoop}{\left[ \begin{array}{cc}
1 & 0 \\
0 & p
\end{array} \right]}

\newcommand{\w}{\omega}
\newcommand{\Qpi}{\ensuremath{\Q(\pi)}}
\newcommand{\Qpin}{\Q(\pi^n)}
\newcommand{\pibar}{\overline{\pi}}
\newcommand{\pbar}{\overline{p}}
\newcommand{\lcm}{\operatorname{lcm}}
\newcommand{\trace}{\operatorname{trace}}
\newcommand{\OKv}{\mathcal{O}_{K_v}}
\newcommand{\Abarv}{\tilde{A}_v}
\newcommand{\kbar}{\overline{k}}
\newcommand{\Kbar}{\overline{K}}
\newcommand{\pl}{\rho_l}
\newcommand{\plt}{\tilde{\pl}}
\newcommand{\plo}{\pl^0}
\newcommand{\Du}{\underline{D}}
\newcommand{\A}{\mathbb{A}}
\newcommand{\D}{\underline{D}}
\newcommand{\op}{\operatorname{op}}
\newcommand{\Glt}{\tilde{G_l}}
\newcommand{\gl}{\mathfrak{g}_l}
\newcommand{\gltwo}{\mathfrak{gl}_2}
\newcommand{\sltwo}{\mathfrak{sl}_2}
\newcommand{\h}{\mathfrak{h}}
\newcommand{\tA}{\tilde{A}}
\newcommand{\sss}{\operatorname{ss}}
\newcommand{\X}{\Chi}
\newcommand{\ecyc}{\epsilon_{\operatorname{cyc}}}
\newcommand{\hatAl}{\hat{A}[l]}
\newcommand{\sA}{\mathcal{A}}
\newcommand{\sAt}{\overline{\sA}}
\newcommand{\OO}{\mathcal{O}}
\newcommand{\OOB}{\OO_B}
\newcommand{\Flbar}{\overline{\F_l}}
\newcommand{\Vbt}{\widetilde{V_B}}
\newcommand{\XX}{\mathcal{X}}
\newcommand{\GbN}{\Gamma_\bullet(N)}
\newcommand{\Gm}{\mathbb{G}_m}
\newcommand{\Pic}{\operatorname{Pic}}
\newcommand{\FPic}{\textbf{Pic}}
\newcommand{\solv}{\operatorname{solv}}
\newcommand{\Hplus}{\mathcal{H}^+}
\newcommand{\Hminus}{\mathcal{H}^-}
\newcommand{\HH}{\mathcal{H}}
\newcommand{\Alb}{\operatorname{Alb}}
\newcommand{\FAlb}{\mathbf{Alb}}
\newcommand{\gk}{\mathfrak{g}_k}
\newcommand{\car}{\operatorname{char}}
\newcommand{\Br}{\operatorname{Br}}
\newcommand{\gK}{\mathfrak{g}_K}
\newcommand{\coker}{\operatorname{coker}}
\newcommand{\red}{\operatorname{red}}
\newcommand{\CAY}{\operatorname{Cay}}
\newcommand{\ns}{\operatorname{ns}}
\newcommand{\xx}{\mathbf{x}}
\newcommand{\yy}{\mathbf{y}}
\newcommand{\E}{\mathbb{E}}
\newcommand{\rad}{\operatorname{rad}}
\newcommand{\Top}{\operatorname{Top}}
\newcommand{\Map}{\operatorname{Map}}
\newcommand{\Li}{\operatorname{Li}}
\renewcommand{\Map}{\operatorname{Map}}
\newcommand{\ZZ}{\mathcal{Z}}
\newcommand{\uu}{\mathfrak{u}}
\newcommand{\mm}{\mathfrak{m}}
\newcommand{\zz}{\mathbf{z}}
\newcommand{\Image}{\operatorname{Image}}
\newcommand{\cc}{\mathfrak{c}}
\newcommand{\ii}{\mathfrak{i}}
\renewcommand{\ss}{\mathfrak{s}}


\noindent
\textbf{Terminology: } Throughout this note, a ``ring'' is a commutative ring 
with multiplicative identity.  A ``domain'' is a ring $R$ such that for all $a,b \in R \setminus \{0\}$, $ab \neq 0$.  A ring $R$ is ``reduced'' if for all 
$x \in R$ and $n \in \Z^+$, if $x^n = 0$ then $x = 0$.    

\section{Introduction}

\subsection{The Combinatorial Nullstellens\"atze}
\textbf{} \\ \\ \noindent
This note concerns the following celebrated results of N. Alon.  

\begin{thm}
\label{11.7.2}
\label{CN}
Let $F$ be a field, let $X_1,\ldots,X_n \subset F$ be nonempty and finite, and $X = \prod_{i=1}^n X_i$.  For $1 \leq i \leq n$, put
\begin{equation}
\label{DIAGONALPHIEQ}
 \varphi_i(t_i) = \prod_{x_i \in X_i} (t_i-x_i) \in F[t_i] \subset F[t] = F[t_1,\ldots,t_n].
\end{equation}
Let $f \in F[t]$ be a polynomial which vanishes on all the common zeros of $\varphi_1,\ldots,\varphi_n$: that is, for all $x \in F^n$, if $\varphi_1(x) = \ldots = \varphi_n(x) = 0$, then $f(x) = 0$.  Then: \\
a) (Combinatorial Nullstellensatz I, or \textbf{CNI}) There are $q_1,\ldots,q_n \in F[t]$ such that
\begin{equation} f(t) = \sum_{i=1}^n q_i(t) \varphi_i(t). 
\end{equation}
b) (Supplementary Relations) Let $R$ be the subring of $F$ generated by the coefficients of $f$ and $\varphi_1,\ldots,\varphi_n$.  Then the $q_1,\ldots,q_n$ may be chosen to lie in $R[t]$ and satisfy
\begin{equation}
\label{SUPPLEMENTEQ}
\forall 1 \leq i \leq n, \ \deg q_i \leq \deg f - \deg \varphi_i. 
\end{equation}
\end{thm}

\begin{thm}(Combinatorial Nullstellensatz II, or \textbf{CNII})
\label{11.7.4}
\label{PM}
Let $F$ be a field, $n \in \Z^+$, $a_1,\ldots,a_n \in \N$, and let $f \in F[t] = F[t_1,\ldots,t_n]$.  We suppose: \\
(i) $\deg f \leq a_1 + \ldots + a_n$.  \\
(ii) The coefficient of $t_1^{a_1} \cdots t_n^{a_n}$ in $f$ is nonzero.  \\
Then, for any subsets $X_1,\ldots,X_n$ of $F$ with $\# X_i = a_i+1$ for $1 \leq i \leq n$, there is $x = (x_1,\ldots,x_n) \in X = \prod_{i=1}^n X_i$ such that 
$f(x) \neq 0$.
\end{thm}
\noindent
Alon used his Combinatorial Nullstellens\"atze to derive various old and new results in number theory and combinatorics, starting with Chevalley's Theorem that a homogeneous polynomial of degree $d$ in at least $d+1$ variables over a finite field has a nontrivial zero.  The use of polynomial methods has burgeoned to a remarkable degree in recent years.  We recommend the recent survey \cite{Tao13}, which lucidly describes the main techniques but also captures the sense of awe and excitement at the extent to which these very simple ideas have cracked open the field of combinatorial number theory and whose range of future applicability seems almost boundless. 
\\ \\
Of Theorems \ref{CN} and \ref{PM}, Theorem \ref{CN} is stronger: 
one easily deduces CNII from CNI and the Supplementary Relations, but (apparently) not conversely.  On the other hand, for appplications in combinatorics and number theory, CNII seems more useful: \cite{Alon99} organizes its applications into seven different sections, and only in the last section is CNI applied.  Later works have followed this trend to an even larger degree, to the exent that most later works simply refer to Theorem \ref{PM} as the Combinatorial Nullstellensatz.  We find this trend somewhat unfortunate: on the one hand, CNI is the stronger result and does have \emph{some} applications in its own right.  On the other hand, it is CNI which is really a Nullstellensatz in the sense of algebraic geometry, and we find this geometric connection interesting and suggestive.
\\ \\
Recently attention has focused on the following sharpening of CNII due to 
Schauz, Lason and Karasev-Petrov \cite[Thm. 3.2]{Schauz08a}, \cite[Thm. 3]{Lason10}, \cite[Thm. 4]{Karasev-Petrov12}).

\begin{thm}(Coefficient Formula)
\label{COEFFTHM}
Let $F$ be a field, and let $f \in F[t]$.  Let 
$a_1,\ldots,a_n \in \N$ be such that $\deg f \leq a_1 + \ldots + a_n$.  For 
each $1 \leq i \leq n$, let $X_i \subset F$ with $\# X_i = a_i + 1$, and let $X = \prod_{i=1}^n X_i$.   Let $d = (a_1,\ldots,a_n)$, and let $c_d$ be the coefficient of $t_1^{a_1} \cdots t_n^{a_n}$ 
in $f$.  Then
\begin{equation}
\label{COEFF}
 c_d = \sum_{x = (x_1,\ldots,x_n) \in X} \frac{f(x)}{\prod_{i=1}^n \varphi_i'(x_i)}.
\end{equation}
\end{thm} 
\noindent
In this note we revisit and further explore these theorems, in three 
different ways:
\\ \\
$\bullet$ In $\S$ 2 we improve CNI to a \textbf{Finitesatz} (Theorem \ref{FRN}): a full Nullstellensatz for polynomial functions on finite subsets of $F^n$ over an arbitrary field $F$.  When $F$ is finite we recover the \textbf{Finite Field Nullstellensatz} of G. Terjanian (Corollary \ref{FFN}). 
\\ \\
$\bullet$ In $\S 3$ we expose the close relation between Theorem \ref{COEFFTHM} and Chevalley's original proof of Chevalley's Theorem.  Adapting Chevalley's method 
gives a version of Theorem \ref{COEFFTHM} valid over \emph{any} ring $R$ 
subject to an additional condition on $X$ which always holds over a field.  This generalization is due to U. Schauz \cite{Schauz08a}.  The main novelty here is our exposition of these results following Chevalley's original arguments.  However, when we close up this circle of ideas we find 
that it yields a \textbf{Restricted Variable Chevalley-Warning Theorem} (Theorem \ref{RESVARCHEV}).  Restricted variable versions of Chevalley's Theorem and Warning's Second Theorem have recently appeared in the literature \cite{Brink11}, \cite{CFS14}, so Theorem \ref{RESVARCHEV} is in some sense the last piece of the ``Chevalley-to-Alon'' conversion process.
\\ \\
$\bullet$ In $\S 4$ we further analyze the evaluation map from 
polynomials to function on an arbitrary subset $X \subset R^n$ for an arbitrary ring.  Our results are far from definitive, and one of our main goals 
of this section is to show that the (perhaps rather arid-looking) formalism of a restricted variable Nullstellensatz leads 
naturally to some interesting open problems in polynomial interpolation over 
commutative rings.


\section{A Nullstellensatz for Finitely Restricted Polynomial Functions}

\subsection{Alon's Nullstellensatz versus Hilbert's Nullstellensatz}
\textbf{} \\ \\ \noindent
The prospect of improving Theorem \ref{CN} \emph{as a Nullstellensatz} has not been explored, perhaps because the notion of a Nullstellensatz, though seminal in algebra and geometry, is less familiar to researchers in combinatorics.  But it was certainly familiar to Alon, who 
began \cite{Alon99} by recalling the following result.

\begin{thm}(Hilbert's Nullstellensatz)
\label{HN}
Let $F$ be an algebraically closed field, let $g_1,\ldots,g_m \in F[t]$, and let $f \in F[t]$ be a polynomial which vanishes 
on all the common zeros of $g_1,\ldots,g_m$.  Then there is $k \in \Z^+$ 
and $q_1,\ldots,q_m \in F[t]$ such that 
\[ f^k = \sum_{i=1}^m q_i g_i. \]
\end{thm}
\noindent
Let us compare Theorems \ref{CN} and \ref{HN}.  They differ in the following points:  \\
$\bullet$ In Theorem \ref{CN}, $F$ can be any field.  In Hilbert's Nullstellensatz, $F$ must be algebraically closed.  Really must: if not, there is a nonconstant 
polynomial $g(t_1)$ without roots in $F$; taking $m = 1$, $g_1 = g$ 
and $f = 1$, the conclusion fails.  \\
$\bullet$ In CNI, the conclusion is that $f$ itself is a 
linear combination of the $\varphi_i$'s with polynomial coefficients, but in 
Hilbert's Nullstellensatz we must allow taking a power of $f$.  Really must: 
e.g. take $k \in \Z^+$ $m = 1$, $g_1 = t_1^k$ and $f = t_1$.  \\
$\bullet$ The Supplementary Relations give upper bounds on the degrees of 
the polynomials $q_i$: they make CNI \textbf{effective}.  Hilbert's Nullstellensatz is not effective.  Effective versions have been given by Brownawell \cite{Brownawell87}, Koll\'ar \cite{Kollar88} and others, but their bounds are much more complicated than the ones in Theorem \ref{CN}.  \\
$\bullet$ In Theorem \ref{CN} the $\varphi_i$'s are extremely restricted.   On the other hand, in Hilbert's Nullstellensatz the $g_i$'s can be any set of polynomials.  Thus Theorem \ref{HN} is a 
\emph{full Nullstellensatz}, whereas Theorem \ref{CN} is a \emph{partial Nullstellensatz}.
\\ \\
We will promote Theorem \ref{CN} to a full Nullstellensatz  
for all finite subsets.

\subsection{The Restricted Variable Formalism}
\textbf{} \\ \\ \noindent
In this section we give the formalism for a Nullstellensatz in the restricted variable context.  Although our main theorem applies to finite subsets of affine $n$-space over a field, it is possible to set things up more generally, and doing so raises some further interesting questions and will be seen to have some useful applications.   
\\ \\
For a set $Z$, let $2^Z$ be the set of all subsets of $Z$.  For a ring $R$, let $\mathcal{I}(R)$ be the set of ideals of $R$.  For a subset $J$ of a ring $R$, let $\langle J \rangle$ denote the ideal of $R$ generated by $J$, and let $\rad J = \rad \langle J \rangle$ denote the set of all $f \in R$ such that $f^k \in \langle J \rangle$ for some $k \in \Z^+$.  An ideal $J$ 
is \textbf{radical} if $J = \rad J$.  
\\ \\
Let $R$ be a ring, and let $X \subset R^n$.  For $x \in X, \ f \in R[t]$, we put 
\[ I(x) = \{f \in R[t] \mid f(x) = 0\}, \]
\[ V_X(f) = \{ x \in X \mid f(x) = 0\}. \]
Put $V = V_{R^n}$.  We may extend $I$ and $V_X$ to maps on power sets as follows:
\[ I: 2^X \ra 2^{R[t]}, \ A \subset X \mapsto I(A) = \bigcap_{a \in A} I(a) = 
\{ f \in R[t] \mid \forall a \in A, \ f(a) = 0\}, \]
\[ V_A: 2^{R[t]} \ra 2^X, \ J \subset R[t] \mapsto V_A(J) = \bigcap_{f \in J} V_A(f) = \{ x \in X \mid \forall f \in J, \ f(a) = 0\}. \]
  Then in fact 
\[ I: 2^X \ra \mathcal{I}(R[t]), \ \forall J \subset R[t], \ V(J) = V(\langle J \rangle). \]
The maps $I$ and $V_A$ are \textbf{antitone}: 
\[A_1 \subset A_2 \subset X \implies I(A_1) \supset I(A_2), \]
\[ J_1 \subset J_2 \subset F[t] \implies V_A(J_1) \supset V_A(J_2), \]
so their compositions are \textbf{isotone}: 
\[ A_1 \subset A_2 \subset X \implies V_X(I(A_1)) \subset V_X(I(A_2)), \]
\[ J_1 \subset J_2 \subset R[t] \implies I(V_X(J_1)) \subset I(V_X(J_2)). \]
We have $X = V_X(0)$, so 
\[\forall J \subset R[t], \ I(V_X(J)) \supset I(V_X(0)) = I(X). \]

\subsection{The Finitesatz}

\begin{lemma}
\label{FRNALEMMA}
a) Suppose $R$ is a domain.  For all ideals $J_1,\ldots,J_m$ of $R[t]$, we have $V_X(J_1 \cdots J_m) = \bigcup_{i=1}^m V_X(J_i)$.  \\
b) Suppose $R$ is reduced.  Then for all $A \subset R^n$, $I(A)$ is a radical ideal.  \\
c) If $R$ is reduced, then for all $J \subset R[t]$, 
\begin{equation}
\label{LOWERBOUNDEQ}
I(V_X(J)) \supset \rad(J+I(X)) \supset \rad J + I(X) \supset J + I(X).
\end{equation}
\end{lemma}
\begin{proof}
a) We intend to allow $m = 0$, in which case the identity reads $V_X(\langle 1 \rangle) = \varnothing$, which is true.  Having established that, we immediately reduce to the case $m = 2$.  Since $J_1 J_2 \subset J_i$ for $i = 1,2$, $V_X(J_1 J_2) \supset V_X(J_i)$ 
for $i = 1,2$, thus $V_X(J_1 J_2) \supset V_X(J_1) \cup V_X(J_2)$.  Now let $x \in X \setminus (V_X(J_1) \cup V_X(J_2))$.  For $i = 1,2$ there 
is $f_i \in J_i$ with $f_i(x) \neq 0$.  Since $R$ is a domain, 
$f_1(x) f_2(x) \neq 0$, so $x \notin V_X(J_1 J_2)$. \\
b) If $f \in R[t]$ and $f^k \in I(A)$ for some $k \in \Z^+$, then for all 
$x \in A$ we have $f(x)^k = 0$.  Since $R$ is reduced, this implies $f(x) = 0$ 
for all $x \in A$ and thus $f \in I(A)$.  \\
c) $I(V_X(J)) = I(X \cap V(J))$ is a radical ideal containing both $I(X)$ and 
$I(V(J)) \supset J$, so it contains $\rad (J+I(X))$.  The other inclusions 
are immediate.  
\end{proof} 
\noindent
It is well known (see Theorem \ref{CATS}) that when $F$ is infinite we have $I(F^n) = \{0\}$.  This serves to motivate the following restatement of Hilbert's Nullstellensatz.

\begin{thm} 
Let $F$ be an algebraically closed field.  For all $J \subset F[t]$, 
\[ I(V(J)) = \rad J. \]
\end{thm}
\noindent
Here is the main result of this section.

\begin{thm}(Finitesatz)
\label{FRN}
\label{MAINTHM}
Let $F$ be a field, and let $X \subset F^n$ be a finite subset. \\  
a) For all ideals $J$ of $F[t]$, we have 
\begin{equation}
\label{FRNAEQ}
 I(V_X(J)) =  J + I(X). 
\end{equation}
In particular, if $J \supset I(X)$ then $I(V_X(J)) = J$. \\
b) (CNI) Suppose $X = \prod_{i=1}^n X_i$ 
for finite nonempty subsets $X_i$ of $F$.  Define $\varphi_i(t_i) \in F[t_i]$ as in (\ref{DIAGONALPHIEQ}) above.   Then 
\begin{equation}
\label{IXEQ}
 I(X) = \langle \varphi_1,\ldots,\varphi_n \rangle. 
\end{equation}
\end{thm}
\begin{proof}
a) Let $F$ be a field, and let $X \subset F^n$ be finite.  Let $x = (x_1,\ldots,x_n) \in X$.  Let $\mm_x = \langle t_1-x_1,\ldots,t_n-x_n \rangle$.  
Then $F[t]/\mm_x \cong F$, so $\mm_x$ is maximal.  On the other hand $\mm_x \subset I(x) \subsetneq F[t]$, so $\mm_x = I(x)$.  Moreover 
$V_X(\mm_x) = \{x\}$, hence 
\[ I(V_X(\mm_x)) = I(x) = \mm_x. \]
Now let $A = \{x_i\}_{i=1}^k \subset X$.  Then
\[ I(A) = I(\bigcup_i \{x_i\}) = \bigcap_i I(x_i) = \bigcap_i \mm_{x_i}, \]
so by the Chinese Remainder Theorem \cite[Cor. 2.2]{Lang},
\[ F[t]/I(A) = F[t]/\bigcap_i \mm_{x_i} \cong \prod_i F[t]/\mm_{x_i} \cong F^{\# X}. \]
Let $F^A$ be the set of all maps $f: A \ra F$, so $F^A$ is an $F$-algebra under pointwise addition and multiplication and $F^A \cong \prod_{i=1}^{\# A} F$.   The \textbf{evaluation map}
\[E_A = F[t] \ra F^A, \ f \in F[t] \mapsto (x \in A \mapsto f(x)) \]
is a homomorphism of $F$-algebras.  Moreover $\Ker E_A = I(A)$, so $E_A$ 
induces a map 
\[ \iota: F[t]/I(A) \hookrightarrow F^A. \]
Thus $\iota$ is an injective $F$-linear map between $F$-vector of equal finite 
dimension, hence an is an isomorphism of rings.  It follows that
\[ \# \mathcal{I}(F[t]/I(X)) = \# \mathcal{I}(F^X) = 2^{\# X}. \]  
By restricting $V_X$ to ideals containing $I(X)$, we get maps
\[ V_X: \mathcal{I}(F[t]/I(X)) \ra 2^X, \]
\[ I: 2^X \ra \mathcal{I}(F[t]/I(X)). \]
For all $A \subset X$, we have 
\[ V_X(I(A)) = V_X(\prod_{i=1}^k \mm_{x_i}) = \bigcup_{i=1}^k V_X(\mm_{x_i}) = \bigcup_{i=1}^k \{x_i\} =  A. \]
Since $\mathcal{I}(F[t]/I(X))$ and $2^X$ have the same finite cardinality, 
it follows that $V_X$and $I$ are mutually inverse bijections!  Thus for any 
ideal $J$ of $F[t]$, using (\ref{LOWERBOUNDEQ}) we get
\[ J+I(X) \subset I(V_X(J)) \subset I(V_X(J+I(X))) = J + I(X). \]
b) Let $d_i = \deg \varphi_i$ and put $\Phi = \langle \varphi_1,\ldots,\varphi_n \rangle$.  Since $\varphi_i|_X \equiv 0$ for all $i$, $\Phi \subset \Ker E$, so there is an induced surjective $F$-algebra homomorphism 
\[ \tilde{E}_X: F[t]/\Phi \ra F[t]/\Ker E_X \ra F^X. \]
Since $F[t]/\Phi$ and $F^X$ are $F$-vector spaces of dimension $d_1 \cdots d_n$,  $\tilde{E}$ is an isomorphism.  Hence $F[t]/\Phi \ra 
F[t]/\Ker E$ is injective, i.e., $\Phi = \Ker E = I(X)$.  
\end{proof}


\begin{cor}(Finite Field Nullstellensatz \cite{Terjanian66})
\label{FFN}
Let $\F_q$ be a finite field.  Then for all ideals $J$ of $\F_q[t]$, we have $I(V_{\F_q^n}(J)) = J + \langle t_1^q-t_1,\ldots,t_n^q-t_n \rangle$.
\end{cor}
\begin{proof}
Apply Theorem \ref{FRN} with $F = X_1 = \ldots = X_n = \F_q$.
\end{proof}

\section{Cylindrical Reduction and the Atomic Formula}

\subsection{From Chevalley to Alon}
\textbf{} \\ \\ \noindent
The first application of CNII in \cite{Alon99} is to Chevalley's Theorem.  But there is a tighter relationship: the technique Alon uses to prove CNI directly generalizes the technique that Chevalley used, a process which we call \textbf{cylindrical reduction}. 
Chevalley applies (his special case of) CNI to prove his theorem in a different way from Alon's deduction of CNII: whereas Alon uses Theorem \ref{CN}b), Chevalley gives an explicit formula for a reduced polynomial in terms of its associated polynomial function.  This \textbf{Atomic Formula} easily implies the Coefficient Formula, which in turn immediately implies CNII.  Moreover, since the spirit of CNII is to deduce information about the coefficients of a polynomial 
from information about its values on a finite set, the Atomic Formula is really the natural result along these lines, as it literally recovers the polynomial 
from its values on a sufficiently large finite set.  Thus we feel that researchers should have the Atomic Formula in their toolkits.
\\ \\
Our Atomic Formula is one of the ``interpolation formulas'' of a 2008 work of U. Schauz \cite[Thm. 2.5]{Schauz08a}.  Unfortunately it seems that Schauz's work has not been properly appreciated. 
Thus in this section we attempt to present this material in a way which reveals it to be as simple and appealing as Chevalley's classic work.  

\subsection{Cylindrical Reduction}
 
\begin{lemma}(Polynomial Division)
\label{POLDIV} \\
Let $R$ be a ring, and let $a(t_1),b(t_1) \in R[t_1]$ with $b$ monic of degree $d$. \\
a) There are unique polynomials $q$ and $r$ with $a = qb + r$ and $\deg r < d$.  \\
b) Suppose $R = A[t_2,\ldots,t_n]$ is itself a polynomial ring over a ring $A$, so $R[t_1] = A[t_1,\ldots,t_n] = A[t]$ and that $b \in A[t_1]$.  Then: \\
$\bullet$ If $q$ has a monomial term 
of multidegree $(d_1,\ldots,d_n)$, then $a$ has a monomial term of 
multidegree $(d_1+d,d_2,\ldots,d_n)$.  It follows that 
\[\deg a \leq \deg q + d. \]
$\bullet$ If $r$ has a monomial term of multidegree $(d_1,\ldots,d_n)$, 
then $a$ has a monomial term of multidegree $(e_1,\ldots,e_n)$ with $d_i \leq e_i$ for all $1 \leq i \leq n$.   It follows that 
\[ \deg r \leq \deg a. \]
\end{lemma}
\begin{proof}
a) Uniqueness: if $a = q_1 b + r_1 = q_2 b + r_2$, then since $b$ is monic and $g_1 \neq g_2$ then we have $d \leq \deg((g_1-g_2)b) = \deg (r_2 - r_1) 
< d$, a contradiction.  Existence: when $b$ is monic, the standard division algorithm involves no division of coefficients so works in any ring.  Part b) follows by contemplating the division algorithm.
\end{proof}

\begin{prop}(Cylindrical Reduction)
\label{CYLINDRICALREDUCTION}
Let $R$ be a ring.  For $1 \leq i \leq n$, let $\varphi_i(t_i) \in F[t_i]$ be monic of degree $d_i$.  Put $\Phi = \langle \varphi_1,\ldots,\varphi_n \rangle$ and $d = (d_1,\ldots,d_n)$.  Say $f \in R[t]$ is \textbf{d-reduced} if for all $1 \leq i \leq n$, $\deg_{t_i} f < d_i$.  Then: \\
a) The set $\mathcal{R}_d$ of all $d$-reduced polynomials is a free $R$-module of rank $d_1 \cdots d_n$.  \\
b) For all $f \in R[t]$, there are $q_1,\ldots,q_n \in R[t]$ such that $\deg q_i \leq 
\deg f - \deg \varphi_i$ and $f-\sum_{i=1}^n q_i \varphi_i$ is $d$-reduced.  \\
c) The composite map $\Psi: \mathcal{R}_d \hookrightarrow R[t] \ra R[t]/\Phi$ is an $R$-module isomorphism.  \\
d) For all $f \in R[t]$, there is a unique $r_d(f) \in \mathcal{R}_d$ such that 
$f-r_d(f) \in \langle \varphi_1,\ldots,\varphi_n \rangle$. 
\end{prop}
\begin{proof}
a) Indeed $\{t_1^{a_1} \cdots t_n^{a_n} \mid 0 \leq a_i < d_i\}$ is a basis for $\mathcal{R}_d$. \\
b) Divide $f$ by $\varphi_1$, then divide 
the remainder $r_1$ by $\varphi_2$, then divide the remainder $r_2$ by $\varphi_n$, and so forth, getting $f = \sum_{i=1}^n q_i \varphi_i + r_n$.  Apply 
Lemma \ref{POLDIV}b).  \\
c) Part b) implies that $\Psi$ is surjective.  For the injectivity: let 
$q_1,\ldots,q_n \in R[t]$ be such that $f = \sum_{i=1}^n q_i \varphi_i \in \mathcal{R}_d$.  We must show that $f = 0$.  For each $i$, by dividing $q_i$ 
by $\varphi_j$ for $i < j \leq n$ and absorbing the quotient into the 
coefficient $q_j$ of $\varphi_j$, we may assume that $\deg_{t_j} q_i < d_j$ 
for all $j > i$.  It now follows easily that for all $1 \leq m \leq n$, 
$\sum_{i=1}^n q_i \varphi_i$ is either $0$ or has $t_i$-degree at least 
$d_i$ for some $1 \leq i \leq m$.  Applying this with $m = n$ shows $f = 0$. \\
d) This follows immediately from part c).  
\end{proof}
\noindent
Let $R$ be a ring, and let $X_1,\ldots,X_n \subset R$ be finite and nonempty.  For $1 \leq i \leq n$ let 
$\varphi_i(t_i)$ be as in (\ref{DIAGONALPHIEQ}) and put $\Phi = \langle \varphi_1,\ldots,\varphi_n \rangle$.  Put $a_i = \# X_i$ and $X = \prod_{i=1}^n X_i$.  We say $f \in R[t]$ is \textbf{X-reduced} 
if it is $(a_1,\ldots,a_n)$-reduced, and we write $\mathcal{R}_X$ for $\mathcal{R}_d$.  We have $\dim \mathcal{R}_d = \prod_{i=1}^n a_i = \# X$.  The 
\textbf{X-reduced representative} of $f$ is the unique polynomial $r_X(f)$ 
such that $f-r_X(f) \in \langle \varphi_1,\ldots,\varphi_n \rangle$.  
\\ \\
Let $S$ be a subset of a ring $R$.  We say $S$ satisfies \textbf{Condition (F)} 
(resp. \textbf{Condition (D)} if for all $x,y \in S$, $x \neq y \implies 
x-y \in R^{\times}$ (resp. $x-y$ is a non-zerodivisor in $R$: if $(x-y)z = 0$ 
then $z = 0$).  Condition (F) implies Condition (D).  Observe that $R$ is a field 
iff every subset satisfies Condition (F), and $R$ is a domain iff every 
subset satisfies Condition (D).  If $X = \prod_{i=1}^n X_i \subset R^n$, we say $X$ satisfies Condition (F) (resp. Condition (D)) if every $X_i$ 
satisfies Condition (F) (resp. Condition (D)).  

\begin{thm}(\textbf{CATS Lemma} \cite{Chevalley35} \cite{Alon-Tarsi92}, \cite{Schauz08a})
\label{CATLEMMA}
\label{CATS}
Let $R$ be a ring.  For $1 \leq i \leq n$, let $X_i \subset R$ be nonempty and finite.  Put $X = \prod_{i=1}^n X_i$. \\
a) (Schauz) The following are equivalent: \\
(i) $X$ satisfies condition (D). \\
(ii) If $f \in \mathcal{R}_X$ and $f(x) = 0$ for all $x \in X$, then $f = 0$. \\
(iii) We have $\Phi = I(X)$.  \\
b) (Chevalley-Alon-Tarsi) The above conditions always hold when $R$ is a domain.
\end{thm}
\begin{proof}
a) (i) $\implies$ (ii): By induction on $n$: suppose $n = 1$.  Write 
$X = \{x_1,\ldots,x_{a_1}\}$, and let $f \in R[t_1]$ have degree less than 
$a_1-1$ such that $f(x_i) = 0$ for all $1 \leq i \leq a_1$.  By Polynomial 
Division, we can write $f = (t_1-x_1)f_2$ for $f_2 \in R[t_1]$.  Since $x_2-x_1$ is not a zero-divisor, 
$f_2(x_2) = 0$, so $f_2(t_1) = (t_1-x_2)$.  Proceeding in this manner we eventually get $f(t_1) = (t_1-x_1) \cdots (t_1 - x_{a_1}) f_{a_1+1}(t_1)$, and comparing 
degrees shows $f = 0$.  Suppose $n \geq 2$ and that the result holds in $n-1$ variables.  Write
\[ f= \sum_{i=0}^{a_n-1} f_i(t_1,\ldots,t_{n-1}) t_n^i\]
with $f_i \in R[t_1,\ldots,t_{n-1}]$.  If $(x_1,\ldots,x_{n-1}) \in \prod_{i=1}^{n-1} X_i$, then $f(x_1,\ldots,x_{n-1},t_n) \in R[t_n]$ has degree less than $a_n$ and vanishes for all $a_n$ elements $x_n \in X_n$, so it is the zero polynomial: $f_i(x_1,\ldots,x_{n-1}) = 0$ for all $0 \leq i \leq a_n$.  By induction, each $f_i(t_1,\ldots,t_{n-1})$ is the zero polynomial and thus $f$ is the zero polynomial. \\
(ii) $\implies$ (iii):  Certainly $\Phi \subset I(X)$.  Let $f \in I(X)$.  
Since $f-r_X(f) \in \Phi \subset I(X)$, for all $x \in X$ we have 
$r_X(f) = f(x) = 0$.  Then (i) gives $r_X(f) = 0$ and thus $f \in \Phi$. \\
(iii) $\implies$ (i): We argue by contraposition: suppose $X$ does not 
satisfy Condition (D).  Then for some $1 \leq i \leq n$, we may write 
$X_i = \{x_1,x_2,\ldots,x_{a_i}\}$ such that there is $0 \neq z \in R$ 
with $(x_1-x_2)z = 0$.  Then $f = z(t_i-x_2)(t_i-x_3) \cdots (t_i-x_{a_i})$
is a nonzero element of $I(X) \cap \mathcal{R}_X$, hence 
$f \in I(X) \setminus \Phi$.  \\
b) If $R$ is a domain then Condition (D) holds for every $X$.  
\end{proof}

\noindent
Suppose $F$ is a domain and $f \in F[t]$ vanishes on $X$.  By Theorem \ref{CATLEMMA} $f \in \Phi$, and thus by Proposition \ref{CYLINDRICALREDUCTION} there are $q_1,\ldots,q_n \in F[t]$ with $\deg q_i \leq \deg f - a_i$ for 
all $i$ such that $f-\sum_{i=1}^n q_i \varphi_i = r_X(f) = 0$, so 
$f = \sum_{i=1}^n q_i \varphi_i$.  This proves Theorem \ref{CN}b).

\subsection{The Atomic Formula and the Coefficient Formula}


\begin{lemma}
Suppose Condition (F).  Let $x = (x_1,\ldots,x_n) \in X$, and put
\[\delta_{X,x} = \prod_{i=1}^n \prod_{y_i \in X_i \setminus \{x_i\}} 
\frac{t_i-y_i}{x_i-y_i} = \prod_{i=1}^n 
\frac{\varphi_i(t_i)}{(t_i-x_i)\varphi_i'(x_i)}
\in F[t]. \]
a) We have $\delta_{X,x}(x) = 1$.  \\
b) If $y \in X \setminus \{x\}$, then $\delta_{X,x}(y) = 0$.  \\
c) For all $1 \leq i \leq n$, $\deg_{t_i} \delta_{X,x} = a_i -1$.  In particular, $\delta_{X,x}$ is $X$-reduced. 
\end{lemma}
\begin{proof} Left to the reader.
\end{proof}

\begin{thm}(Atomic Formula) Suppose Condition (F).  For all $f \in R[t]$, we have 
\begin{equation}
\label{ATOMICDECOMPOSITION}
r_X(f) = \sum_{x \in X} f(x) \delta_{X,x}.
\end{equation}
\end{thm}
\begin{proof} 
Apply Theorem \ref{CATLEMMA}a) to $r_X(f) - \sum_{x \in X} f(x) \delta_{X,x}$.
\end{proof}
\noindent
Let $d = (d_1,\ldots,d_n) \in \N^d$.  We say a polynomial $f \in F[t]$ is 
\textbf{d-topped} if for any $e = (e_1,\ldots,e_n)$ with $d_i \leq e_i$ 
for all $1 \leq i \leq n$ and $\sum_{i=1}^n d_i < \sum_{i=1}^n e_i$, 
then the coefficient of $t^e = t_1^{e_1} \cdots e_n^{e_n}$ in $f$ is $0$.  

\begin{remark}
\label{DEGREEDTOPPED}
If $\deg f \leq d_1 + \ldots + d_n$, then $f$ is $d$-topped.  
\end{remark}

\begin{lemma}
\label{TOPPEDLEMMA}
Let $d = (a_1-1,\ldots,a_n-1)$, and let $f \in F[t]$ be $d$-topped.  
Then the coefficient of $t^d = t_1^{a_1-1} \cdots t_n^{a_n-1}$ in $f$ is 
equal to the coefficient of $t^d$ in $r_X(f)$.   
\end{lemma}
\begin{proof}
Write $\varphi_i(t_i)= t_i^{a_i} - \psi_i(t_i)$, $\deg(\psi_i) < a_i$.  An \textbf{elementary cylindrical reduction} of $f$ consists of identifying a monomial which is divisible by $t_i^{a_i}$ and replacing $t_i^{a_i}$ by $\psi_i(t_i)$.  
Elementary cylindrical reduction on a $d$-topped polynomial 
yields a $d$-topped polynomial with the same coefficient of $t^d$.  The reduced polynomial $r_X(f)$ is obtained from $f$ by finitely many elementary cylindrical reductions.
\end{proof}
\noindent
We deduce a \emph{proof} of the Coefficient Formula (Theorem \ref{COEFFTHM}): by Remark \ref{DEGREEDTOPPED}, $f$ is $d$-topped, so 
$c_d(f) = c_d(r_X(f))$.   Apply (\ref{ATOMICDECOMPOSITION}).  


\begin{remark}
The proof shows that Theorem \ref{COEFFTHM} holds with weaker hypotheses: \\
(i) ``$\deg f \leq a_1 + \ldots + a_n$'' can be weakened to ``$f$ is $(a_1,\ldots,a_n)$-topped''.  \\
(ii)  ``$F$ is a field'' can be weakened to ``$S$ satisfies Condition (F)''.  It can be further weakened to ``$S$ satisfies 
Condition (D)'' so long as we interpret (\ref{COEFF}) as taking place 
in the total fraction ring of $F$ (equivalently, if we clear denominators).  \\
The first strengthening appears in the work of Schauz and Las\'on and 
the second appears in the work of Schauz (see especially \cite[Thm. 2.9]{Schauz08a}).  
\end{remark}


\subsection{The Restricted Variable Chevalley-Warning Theorem}
\textbf{} \\ \\ \noindent
For a ring $R$ and $x = (x_1,\ldots,x_n) \in R^n$, we put $w(x) = \# \{1 \leq i \leq n \mid x_i \neq 0 \}$.

\begin{thm}(Restricted Variable Chevalley-Warning Theorem)  \label{RESVARCHEV}
Let $P_1,\ldots,P_r \in \F_q[t] = \F_q[t_1,\ldots,t_n]$ be polynomials of 
degrees $d_1,\ldots,d_r$.  For $1 \leq i \leq n$, let $\varnothing \neq X_i 
\subseteq \F_q$ be subsets, put $X = \prod_{i=1}^n X_i$ and also 
\[V_X = \{ x = (x_1,\ldots,x_n) \in X \mid 
P_1(x) = \ldots = P_r(x) = 0 \}. \]
Suppose that $(d_1+\ldots+d_r)(q-1) < \sum_{i=1}^n \left( \# X_i - 1 \right)$.  Then: \\
a) As elements of $\F_q$, we have
\begin{equation}
\label{RVCWEQ1}
 \sum_{x \in V_X} \frac{1}{\prod_{i=1}^n \varphi_i'(x_i)} = 0 
\end{equation}
and thus \cite{Schauz08a} \cite{Brink11} 
\begin{equation}
\label{RVCWEQ2}
 \# V_X \neq 1. 
\end{equation}
b) (Chevalley-Warning) If $\sum_{i=1}^r d_i < 
n$, then $p \mid \# \{ x \in \F_q^n \mid P_1(x) = \ldots = P_r(x) \}$.  \\
c) (Wilson) If $(d_1 + \ldots + d_r)(q-1) < n$, then 
\[ \# \{ x \in V_{ \{0,1\}^n} \mid  w(x) \equiv 0 \pmod{2}\} \equiv \# \{ x \in V_{\{0,1\}^n} \mid w(x) \equiv 1 \pmod{2}\} \pmod{p}. \]
d) If $(d_1+\ldots+d_r)(q-1) < (q-2)n$, then 
\[ \sum_{x \in \F_q^n \mid f_1(x) = \ldots = f_r(x) = 0} x_1 \cdots x_n = 0. \]
\end{thm}
\begin{proof}
a) We define
\[
P(t) = \chi_{P_1,\ldots,P_r}(t) =  \prod_{i=1}^r \left(1-P_i(t)^{q-1}\right),
\]
so $\deg P = (q-1)(d_1+\ldots+d_r) < \sum_{i=1}^n \left( \# X_i - 1 \right)$, and thus the coefficient of $t_1^{\# X_1-1} \cdots 
t_n^{\# X_n-1}$ in $P$ is $0$.  Applying the Coefficient Formula, we get
\[ 0 = \sum_{x \in X}  \frac{P(x)}{\prod_{i=1}^n \varphi_i'(x_i)} = 
 \sum_{x \in V_X} \frac{1}{\prod_{i=1}^n \varphi_i'(x_i)} \in \F_q. \]
Parts b) through d) follow from part a) by taking $X$ to be, respectively, 
$\F_q^n$, $\{0,1\}^n$ and $(\F_q^{\times})^n$, and computing the $\varphi_i'(t_i)'s$.  The details are left to the reader.
\end{proof}

\section{Further Analysis of the Evaluation Map}

\subsection{The Finitesatz holds only over a field}
\textbf{} \\ \\ \noindent
If $R$ is a ring which is not a field and $X \neq \varnothing$, then the assertion of Theorem \ref{FRN}a) remains meaningful with $R$ in place of $F$, but it is false.  
Let $x \in X$. Since $F[t]/\mm_x \cong F$, $\mm_x$ 
is \emph{not} maximal, so let $J$ be an ideal with $\mm_x \subsetneq J \subsetneq F[t]$, and let $f \in J \setminus \mm_x$.  Then $V_X(J) \subset V_X(\mm_x) = \{x\}$, 
and since $f \notin \mm_x$, $f(x) \neq 0$.  So
\[I(V_X(J)) = I(\varnothing) = F[t] \supsetneq J = J + I(X). \]

\subsection{Towards an Infinitesatz}
\textbf{} \\ \\ \noindent
We revisit the formalism of $\S$ 2.2: let $R$ be a ring any $X \subset R^n$.
\\ \\
For a subset $A \subset R^n$ we define the \textbf{Zariski closure} $\overline{A} = V(I(A))$.  Thus $\overline{A}$ is the set of points 
at which any polynomial which vanishes at every point of $A$ must also vanish.  A subset $A$ is \textbf{algebraic} if $A = \overline{A}$ and \textbf{Zariski-dense} 
if $\overline{A} = R^n$.  When $R$ is a domain 
the algebraic subsets are the closed sets of a topology, the \textbf{Zariski topology}.
this need not hold and some strange things can happen: for instance if $R = \Z/6\Z$ and $n = 1$ then $\overline{ \{2,3\} } 
= \{0,2,3,5\}$.  
\\ \\
If $F$ is an algebraically closed field and $X \subset F^n$ is algebraic, then using Hilbert's Nullstellensatz, for all ideals $J$ of $F[t]$, 
\[ I(V_X(J)) = I(V(J) \cap X) = I(V(J) \cap V(I(X))) \] \[ =I(V(J \cup I(X))) = I(V(J+I(X))) = \rad (J + I(X)). \]
When $X$ is infinite, we {\sc claim} the ``$\rad$'' cannot be removed.  
\\
\noindent
{\sc proof of claim}: Suppose $\rad (J + I(X)) = J + I(X)$ for all $J$.  Equivalently, every ideal 
$J \supset I(X)$ is a radical ideal.  Then for any element $x$ the quotient ring $F[t]/I(X)$, since $(x^2)$ is radical we must have $(x) = (x^2) = (x)^2$.  It follows that $F[t]/I(X)$ is absolutely flat, hence has Krull dimension zero, 
hence Artinian, hence has finitely many maximal ideals.  Since $x \mapsto \mm_x$ is an injection from $X$ to the set of maximal ideals of $F[t]/I(X)$, $X$ is finite.  
\\ \\
The case of an arbitrary subset over an arbitrary ring $R$ is much more challenging.  In fact, even determining whether the evaluation map $E_X: R[t] \ra R^X$ is surjective -- \textbf{existence of interpolation polynomials} -- or injective -- \textbf{uniqueness of interpolation polynomials} --
becomes nontrivial.  In the next section we address these questions, but we are not able to resolve them completely.  

\subsection{Injectivity and Surjectivity of the Evaluation Map}

\begin{lemma}
\label{LAMLEMMA}
Let $R$ be a ring.  Let $M_1$ and $M_2$ be free $R$-modules, with bases 
$\mathcal{B}_1$ and $\mathcal{B}_2$.  If $\iota: M_1 \ra M_2$ is an injective $R$-module map, then $\# \mathcal{B}_1 \leq \# \mathcal{B}_2$.
\end{lemma}
\begin{proof}
Combine \cite[Cor. 1.38]{LMR} and \cite[Ex. 1.24]{EMR}.
\end{proof}

\begin{lemma}
\label{LINEARALGEBRAFACT}
Let $R$ be a ring, and let $X$ be an infinite set.  Then $R^X$ is \emph{not} a countably generated $R$-module.
\end{lemma}
\begin{proof}
Step 1: For $x \in \R$, let $A_x = \{ y \in \Q \mid y < x\}$, and let $\mathcal{C}_{\Q} = \{ A_x\}_{x \in \R}$.  Then $\mathcal{C}_{\Q} \subset 2^{\Q}$ is an uncountable linearly ordered family of nonempty subsets of $\Q$.  Since $X$ is infinite, 
there is an injection $\iota: \Q \hookrightarrow X$; then $\mathcal{C} = \{ \iota(A_x)\}_{x \in \R}$ is an uncountable linearly ordered family of nonempty subsets of $X$.  
\\ 
Step 2: For each $A \in \mathcal{C}$, let $1_A$ be the characteristic function of $A$.  Then $\{1_A\}_{A \in \mathcal{C}}$ is an $R$-linearly independent set: let $A_1,\ldots,A_n \in \mathcal{C}$ and $\alpha_1,\ldots,\alpha_n \in R$ be such that $\alpha_1 1_{A_1} + \ldots + \alpha_n 1_{A_n} \equiv 0$.  We may order the $A_i$'s such that $A_1 \subset \ldots \subset A_n$ and thus there is $x \in A_n \setminus \bigcup_{i=1}^{n-1} A_{i}$.  Evaluating at $x$ gives $\alpha_n = 0$.  In a similar manner we find that $\alpha_{n-1} = \ldots = \alpha_1 = 0$.
\\ 
Step 3: Suppose $R^X$ is countably generated: thus there is a surjective $R$-module map $\Phi: \bigoplus_{i=1}^{\infty} R \rightarrow R^X$.  For each $A \in \mathcal{C}$, choose $e_A \in \Phi^{-1}(1_A)$ and put $\mathcal{S} = \{ e_A \mid A \in \mathcal{C} \}$.  By Step 2, $\mathcal{S}$ is uncountable and $R$-linearly independent, so it spans a free $R$-module with an uncountable basis which is 
an $R$-submodule of $\bigoplus_{i=1}^{\infty} R$, contradicting Lemma \ref{LAMLEMMA}.
\end{proof}

\begin{thm}
\label{EVALSURJ}
If $X \subset R^n$ is infinite, then $E_X: R[t] \ra R^X$ is \emph{not} surjective.
\end{thm}
\begin{proof}
If $E_X: R[t] \ra R^X$ were surjective, then $R^X$ would be a countably generated $R$-module, contradicting Lemma \ref{LINEARALGEBRAFACT}.
\end{proof}
\noindent
If $Y \subset X \subset R^n$, restricting functions from $X$ 
to $Y$ gives a surjective $R$-algebra map $\mathfrak{r}_Y: R^X \ra R^Y$.  
Moreover, $E_Y = \mathfrak{r}_Y \circ E_X$.  Thus if $E_X$ is surjective, so is 
$E_Y$.
\\ \\
Let $\pi_i: R^n \ra R$ be the $i$th projection map: $\pi_i: (x_1,\ldots,x_n) \mapsto x_i$.  For a subset $X \subset R^n$, we define the \textbf{cylindrical hull} $\mathcal{C}(X)$ as $\prod_{i=1}^n \pi_i(X)$: it is the unique minimal cylindrical subset containing $X$, and it is finite iff $X$ is.  

\begin{prop}
\label{PARTIALSURJ}
Let $X \subset R^n$ be finite.  \\
a) If $\mathcal{C}(X)$ satisfies Condition (F), then $E_X$ is surjective.  \\
b) If there is a nonempty cylindrical subset $Y = \prod_{i=1}^n Y_i \subset X$ which does not satisfy Condition (F), then $E_X$ is not surjective.
\end{prop}
\begin{proof}
a) Since $X \subset \mathcal{C}(X)$, it suffices to show that $E_{\mathcal{C}(X)}$ 
is surjective, and we have essentially already done this: under Condition (F) we may \emph{define} $r_X(f) = \sum_{x \in X} f(x) \delta_{X,x}(t)$, and as in $\S$ 3.3 we see that $E(r_X(f)) = f$.  \\
b) There is $1 \leq i \leq n$ and $y_i \neq y_i' \in Y_i$ such that 
$y_1-y_2 \notin R^{\times}$, hence a maximal ideal $\mm$ of 
$R$ with $y_1-y_2 \in \mm$.  For all $j \neq i$, choose $y_j \in Y_j$; let 
$y = (y_1,\ldots,y_n)$; and let $y'$ be obtained from $y$ by changing the 
$i$th coordinate to $y_i'$.  For any $f \in F[t]$, $f(y) \equiv f(y') \pmod{m}$, 
so $f(y) - f(y') \in \mm$.  Hence the function $\delta_{Y,y}: Y \ra R$ which maps $y$ to $1$ and every other element of $Y$ to $0$ does not lie in the image of the evaluation map.  Thus $E_Y$ is not surjective, so $E_X$ cannot be surjective. 
\end{proof}
\noindent
Thus if $X$ is itself cylindrical, the evaluation map is surjective iff $X$ satisfies Condition (F): this result is due to Schauz.  Proposition \ref{PARTIALSURJ} is the mileage one gets from this in the general case.  When every cylindrical subset of $X$ 
satisfies condition (F) but $\mathcal{C}(X)$ does not, the question of the existence of interpolation polynomials is left open, to the best of my knowledge even e.g. over $\Z$.    
\\ \\
We say that a ring $R$ is \textbf{(F)-rich} (resp. \textbf{(D)-rich}) if for all $d \in \Z^+$ there is a $d$-element subset of $R$ satisfying Condition (F) (resp. Condition (D)).  If $\iota: R \hookrightarrow S$ is a ring embedding and $R$ is 
(F)-rich, then $S$ is (F)-rich, hence also (D)-rich.  

\begin{prop}
\label{PROP14}
Let $R$ be a ring and $X \subset R^n$.  Consider the following assertions: \\
(i) $E_X$ is injective.  \\
(ii) $X$ is infinite and Zariski-dense.  \\
a) We always have (i) $\implies$ (ii).  \\
b) If $R$ is (D)-rich -- e.g. if it contains an (F)-rich subring -- then (ii) $\implies$ (i). \\
c) If $R$ is finite, a domain, or an algebra over an infinite field, then (ii) $\implies$ (i). \\
d) If $R$ is an infinite Boolean ring -- e.g. $R = \prod_{i=1}^{\infty} \Z/2\Z$ -- and $X = R^n$, then (ii) holds and (i) does not.
\end{prop}
\begin{proof}
a) By contraposition: suppose first that $X$ is finite.  Then 
$F^X$ is a free $F$-module of finite rank $\# X$ and $F[t]$ is a free $F$-module of infinite rank, so $E$ cannot be injective.  Now suppose $X$ is not 
Zariski dense: then there is 
$y \in F^n \setminus X$ and $f \in F[t]$ such that $E(f)|_X \equiv 0$ 
and $E(f)(y) \neq 0$, hence $0 \neq f \in \Ker E$.  \\
b) Let $f \in \Ker E_X = I(X)$, and let $d = \deg f$.  Since $X$ is Zariski-dense in $F^n$, $f(x) = 0$ for all $x \in F^n$.  Since $R$ is (D)-rich, there is a $S \subset R$ of cardinality $d+1$ satisfying Condition (D).  Put $X = \prod_{i=1}^n S$.  
Then $f \in \mathcal{R}_X$ and $f(x) = 0$ for all $x \in X$, so $f = 0$ by Theorem \ref{CATS}. \\
c) This is immediate from part b). \\
d) Since $R$ is infinite, $R^n$ is infinite and Zariski-dense.  Since $R$ is Boolean, the polynomial $t_1^2-t_1$ evaluates to zero on every $x \in R^n$.  
\end{proof}
\noindent
\textbf{Acknowledgments}: My interest in Combinatorial Nullstellens\"atze and connections to Chevalley's Theorem was kindled by correspondence with John R. Schmitt.  The main idea for the proof of Lemma \ref{LINEARALGEBRAFACT} is due to Carlo Pagano.  I thank Emil Je\v r\'abek for introducing me to the Finite 
Field Nullstellensatz.

\end{document}